\documentclass[12pt]{amsart}
\usepackage[latin1]{inputenc}
\usepackage{color}
\usepackage{bbm}
\usepackage{amsmath,amssymb}
\usepackage{enumerate}
\usepackage{mathrsfs}
\usepackage{verbatim}

\newtheorem{thm}{Theorem}[section]
\newtheorem{cor}[thm]{Corollary}
\newtheorem{lem}[thm]{Lemma}
\newtheorem{prop}[thm]{Proposition}
\theoremstyle{definition}
\newtheorem{defin}[thm]{Definition}

\theoremstyle{remark}
\newtheorem*{remark}{Remark}

\numberwithin{equation}{section}

\makeatletter
\newcommand{\supp}{\operatorname{supp}}

\newcommand{\dif}{\,\mathrm{d}}

\newcommand{\charfun}{\ensuremath{\mathbbm 1}}

\DeclareMathOperator{\conv}{re}

\allowdisplaybreaks

\begin{document}
\title{Martingale  convergence Theorems for Tensor Splines}

\author[M. Passenbrunner]{Markus Passenbrunner}
\address{Institute of Analysis, Johannes Kepler University Linz, Austria, 4040 Linz, Alten\-berger Strasse 69}
\email{markus.passenbrunner@jku.at}


\begin{abstract}
	In this article we prove martingale type pointwise convergence theorems pertaining to tensor
	product splines defined on $d$-dimensional Euclidean space ($d$ is a positive
	integer),
	where conditional expectations are replaced by their corresponding tensor spline
	orthoprojectors.
	Versions of Doob's maximal inequality, the martingale convergence
	theorem and the characterization of the Radon-Nikod\'{y}m property of
	Banach spaces $X$ in terms of pointwise $X$-valued martingale convergence are obtained
	in this setting.
	Those assertions are in full analogy to their
	martingale counterparts and hold independently of filtration, spline
	degree, and 
	dimension~$d$.
\end{abstract}

\maketitle 

\allowdisplaybreaks
\section{Introduction}
In this article we prove pointwise convergence theorems pertaining to tensor
product splines defined on $d$-dimensional Euclidean space in the spirit of 
the known results for martingales. We begin by discussing the
situation for martingales and, subsequently, for one-dimensional splines. 
For martingales, we use \cite{Neveu1975} and \cite{DiestelUhl1977} as
references.
Let $(\Omega, (\mathscr F_n), \mathbb P)$ be a filtered probability space. A
sequence of integrable functions $(f_n)_{n\geq 1}$ is a \emph{martingale} if $\mathbb
E(f_{n+1} | \mathscr F_n) = f_n$ for any $n$, where we denote by $\mathbb
E(\cdot | \mathscr F_n)$ the conditional expectation operator with respect to
the $\sigma$-algebra $\mathscr F_n$. This operator is the orthoprojector onto
the space of $\mathscr F_n$-measurable $L^2$-functions and it 
can be extended to act on the Lebesgue-Bochner space
$L^1_X$ for  any Banach space $(X,\|\cdot\|)$.
Observe that if $f\in L^1_X$, the sequence  $(\mathbb E(f|\mathscr F_n))$
is a martingale.
In this case, we have that $\mathbb E(f|\mathscr F_n)$ converges
almost surely to $\mathbb E(f|\mathscr F)$ with $\mathscr F = \sigma(
\cup_n\mathscr F_n)$. A~crucial step in the proof of this
convergence theorem is 
\emph{Doob's maximal inequality}
		\begin{equation*}
			\mathbb P( \sup_n \|f_n\| > t ) \leq 
		\frac{\sup_n\|f_n\|_{L^1_X}}{ t}, \qquad t>0,
		\end{equation*}
which states that the martingale maximal function $\sup \|f_n\|$ is of weak type $(1,1)$.
For general scalar-valued martingales, we have the following convergence
theorem: 
 any martingale $(f_n)$ that is bounded in $L^1$ has
an almost sure limit function contained in $L^1$. This limit can be identified
as the Radon-Nikod\'{y}m derivative of the $\mathbb P$-absolutely continuous part 
of the measure
$\nu$ defined by
\begin{equation}\label{eq:martingalnu}
	\nu (A) = \lim_m \int_A f_m \dif\mathbb P,\qquad A\in \cup_n \mathscr F_n.
\end{equation}
This limit exists because of the martingale property of $(f_n)$.
The same convergence theorem as
above holds true 
for $L^1_X$-bounded $X$-valued martingales $(f_n)$, 
provided there exists a
Radon-Nikod\'{y}m derivative of the $\mathbb P$-absolutely continuous part
of the now $X$-valued measure $\nu$ in \eqref{eq:martingalnu}.
Banach spaces $X$ where this is always possible are said to have the
\emph{Radon-Nikod\'{y}m property} (RNP) (see Definition~\ref{def:rnp}).
The RNP of a Banach space is even characterized by martingale
convergence meaning that in any Banach space $X$ without RNP, we can find a
non-convergent and $L^1_X$-bounded martingale.

Consider now the special case where each $\sigma$-algebra $\mathscr F_n$ is generated by
a partition of a bounded interval $I\subset \mathbb R$ into finitely many 
subintervals $(I_{n,i})_i$ of positive length as atoms of $\mathscr F_n$. In this case, $(\mathscr F_n)$ is
called an \emph{interval filtration} on $I$. Then, the
characteristic functions $(\charfun_{I_{n,i}})$ of those atoms are a sharply localized
orthogonal basis of $L^2(\mathscr F_n)$ w.r.t. (with respect to) Lebesgue measure
$\lambda=|\cdot|$. 
If we want to preserve the localization property of the basis functions, but 
at the same time consider spaces of functions with higher smoothness, a natural
candidate is the space of piecewise polynomial functions of order $k$, given by
\begin{align*}
	S^{k}(\mathscr F_n) = \{f : I\to \mathbb R\ |\ &\text{$f$ is $k-2$ times continuously
	differentiable and} \\
	&\qquad\text{a polynomial of order $k$ on
	each atom of $\mathscr F_n$} \},
\end{align*}
where $k$ is an arbitrary positive integer.
One reason for this is that $S^k(\mathscr F_n)$ admits a special basis, the so
called 
B-spline basis $(N_{n,i})_i$, that consists of non-negative and localized
functions $N_{n,i}$. 
Here, the term ``localized'' means that the support of each function $N_{n,i}$
consists of at most $k$ neighbouring atoms of $\mathscr F_n$.
A second reason is that if $(\mathscr F_n)$ is an increasing sequence of
interval $\sigma$-algebras, then the sequence of corresponding spline spaces $S^k(\mathscr
F_n)$ is increasing as well. 
Note that the aforementioned properties of the B-spline functions $(N_{n,i})$ imply that they
do not form an orthogonal basis of $S^k(\mathscr F_n)$ for $k\geq 2$.
For more information on spline functions,
see e.g. \cite{Schumaker2007}.
Let $P_n^{k}$ be the orthogonal projection operator onto
$S^{k}(\mathscr F_n)$ with respect to the
$L^2$  inner product on $I$  equipped with the Lebesgue measure.  
Since the space $S^{1}(\mathscr F_n)$ consists of
piecewise constant functions,
 $P_n^{1}$ is the conditional expectation
operator with respect to the $\sigma$-algebra $\mathscr F_n$ and the Lebesgue
measure.
In general, the operator $P_n^{k}$ can be written in terms of the B-spline basis $(N_{n,i})$ as
\begin{equation}\label{eq:Pn}
	P_n^{k} f = \sum_i \int_I f N_{n,i}\dif\lambda \cdot N_{n,i}^*,
\end{equation}
where the functions $(N_{n,i}^*)$,
contained in the spline space $S^{k}(\mathscr F_n)$, are the biorthogonal (or
dual) system to the
B-spline basis $(N_{n,i})$. 
 Due to the uniform boundedness of the B-spline functions
 $N_{n,i}$, we are able to insert functions $f$ in formula
\eqref{eq:Pn} that are contained not only in $L^2$, but in the
Lebesgue-Bochner space $L^1_X$, thereby extending the operator
$P_n^{k}$ to $L^1_X$.

 Similarly to the definition of martingales, we adopt the following
 notion introduced in \cite{Passenbrunner2020}:
let $(f_n)_{n\geq 1}$ be a sequence of functions in the space $L^1_X$. We call this
sequence a \emph{$k$-martingale spline sequence} (adapted to $(\mathscr F_n))$
if
\[
	P_n^{k} f_{n+1} = f_n,\qquad n\geq 1.
\]
The local nature of the B-splines and the nestedness of the spaces
$(S^k(\mathscr F_n))_n$ ultimately allow us to 
 transfer the classical martingale theorems discussed above
	to
	$k$-martingale spline sequences adapted to  \emph{arbitrary}
	interval filtrations ($\mathscr F_n$) and for any positive integer $k$, just by replacing
	conditional expectation operators with the spline projection operators
	$P_n^{k}$. Indeed, for any positive integer $k$, we have the following results.
		\begin{enumerate}[(i)]
			\item\label{it:splines1} (Shadrin's theorem) \\
				There exists a constant $C$ (depending only on
				$k$ and not on $(\mathscr F_n)$) such that 
				\[
					\sup_n\| P_n^{k} : L^1_X \to L^1_X \| \leq C.
				\]
			\item \label{it:splines2}
					(Doob's inequality for
					splines) \\
				There exists a constant $C$
				such that for any $k$-martingale
				spline sequence
				$(f_n)$,
				\begin{equation*}
					|\{ \sup_n \|f_n\| > t \}| \leq C
				\frac{\sup_n\|f_n\|_{L^1_X}}{ t}, \qquad t>0.
				\end{equation*}

			\item\label{it:splines3}
				(Pointwise convergence of spline projections)\\
				For any Banach space $X$ and any $f\in L^1_X$,
				the sequence 
				$(P_n^{k} f)$ converges
				 almost everywhere to some $L^1_X$-function.
				 
				\item \label{it:splines4}
					(RNP characterization by pointwise spline
					convergence) \\
					For any Banach space $X$,
					the following statements are equivalent:
					\begin{enumerate}
						\item $X$ has RNP,
						\item every $k$-martingale
							spline sequence that is bounded
							in $L^1_X$ converges almost
							everywhere to an
							$L^1_X$-function.
					\end{enumerate}
		 \end{enumerate}
	We give a few comments
	regarding the proofs of the statements (i)--(iv) above.
	Property (i), for arbitrary $k$,
	was proved by A. Shadrin in the groundbreaking paper \cite{Shadrin2001}.  We also
	refer to the article \cite{Golitschek2014} by M. v.~Golitschek, who gave 
	a substantially shorter proof of \eqref{it:splines1}. It should be noted
	that in the case $k=1$, due to Jensen's inequality for conditional
	expectations, we can choose $C=1$ in \eqref{it:splines1}.
	Property \eqref{it:splines2} was proved in
	\cite{PassenbrunnerShadrin2014}. By a standard argument for passing from
	a weak type (1,1) inequality of a maximal function to a.e. convergence
	for $L^1_X$-functions (see for instance Chapter~1 of \cite{Garsia1970}), 
	item \eqref{it:splines3} was proved
	in \cite{PassenbrunnerShadrin2014} in the case that $\cup_n \mathscr
	F_n$ generates the Borel-$\sigma$-algebra on $I$ and in 
	 \cite{MuellerPassenbrunner2020} in general. We also identify the limit
	 of $P_n^k f$
	 as $P_\infty f$, where $P_\infty$ is (the $L^1_X$-extension of) the orthogonal
				 projector onto the
				 $L^2$-closure of $\cup_n S^{k}(\mathscr
				 F_n)$.
	The implication (a)$\implies$(b) in item \eqref{it:splines4} was also
	proved in \cite{MuellerPassenbrunner2020}, whereas the reverse
	implication (b)$\implies$(a) was shown in \cite{Passenbrunner2020} by constructing a
	non-convergent, $L^\infty_X$-bounded $k$-martingale spline sequence with values in Banach
	spaces $X$ without RNP for any positive integer $k$.

Almost everywhere convergence of orthogonal Franklin (i.e. the piecewise linear case) and spline series has 
	a long history:
	It was proved by Z. Ciesielski in \cite{Ciesielski1975}
	that  orthonormal spline expansions of $L^1$-functions with respect to 
	the dyadic partition on the interval 
	converge almost everywhere.  Z. Ciesielski and A. Kamont \cite{CiesielskiKamont1997}
	showed that Franklin series of integrable functions corresponding
	to arbitrary partitions converge almost everywhere for every possible partition.

	Let us also mention a few results in a slightly different direction. Note that in the 
	underlying manuscript, we consider the question under which conditions 
	martingale spline sequences $(f_n)$ converge almost everywhere to \emph{some} function 
	$f$.   As the $L^1$-bounded martingale $f_n = 2^n \charfun_{[0,2^{-n}]}$ on the unit interval shows, 
	the a.e. limit $f$ (which is zero in this case) 
	cannot be used to recover the sequence $(f_n)$. One can then ask for conditions so that 
	$(f_n)$ actually is (uniquely) determined by the limit function $f$.
Such conditions  were given for the piecewise linear dyadic case 
	by G.~G.~Gevorkyan \cite{Gevorkyan1989}. 
	There are many generalizations of this 
	result, see for instance M.~Pohosyan \cite{Poghosyan2000} or G.~G.~Gevorkyan, K.~A.~Navasardyan \cite{GevorkyanNavasardyan2018} for more general 
	partitions, G.~G.~Gevorkyan \cite{Gevorkyan2017multi} for the multivariate case,
	and K.~Keryan, A.~Khachatryan \cite{Keryan2023} for higher order splines.
	For more information regarding such so-called uniqueness results,
		 the interested reader should consult the references cited in the aforementioned articles.

	 In this article we are concerned with pointwise convergence of 
	 multivariate martingale spline sequences. 
	 Let $d$ be a positive integer and, for $j=1,\ldots,d$, let $(\mathscr
	 F_n^{j})$ be an interval filtration on the interval $I\subset \mathbb R$. Filtrations $(\mathscr
	 F_n)$ of the form $\mathscr F_n = \mathscr F_n^1\otimes \cdots \otimes
	 \mathscr F_n^d$ will be called an \emph{interval filtration} on the
	 cube $I^d$. Then, the atoms of $\mathscr F_n$ are of the form
	 $A_1\times \cdots\times A_d$ with atoms $A_j$ in $\mathscr F_n^j$.
	 For a tuple $k=(k_1,\ldots,k_d)$ consisting of $d$ positive integers,
	 denote by $P_n^{k}$ 
	 the orthogonal projector with respect to
	 $d$-dimensional Lebesgue measure $|\cdot|=\lambda^d$ onto the tensor product
	 spline space $S^{k_1}(\mathscr F_n^{1}) \otimes \cdots \otimes
	 S^{k_d}(\mathscr F_n^{d})$.
	 The tensor product structure of $P_n^{k}$ immediately allows us to
	 conclude (i) in this case, i.e., $P_n^{k}$ is bounded on
	 $L^1_X(I^d)$ by a constant depending only on $k$ (cf. also
	 \cite[Corollary~3.1]{PassenbrunnerProchno2019}).

	 Similarly to the one-dimensional case above, we then introduce the following notion:
	 \begin{defin} Let $(\mathscr F_n)$ be an interval filtration on
		 a $d$-dimensional cube $I^d$. A sequence of functions $(f_n)_{n\geq
	 1}$ in the space $L^1_X(I^d)$ is a
	 \emph{$k$-martingale spline sequence} (adapted to ($\mathscr F_n$)) if
	 \[
		P_n^k f_{n+1} = f_n, \qquad n\geq 1.
	 \]
	\end{defin}
	 The implication (b)$\implies$(a) in item (iv) for martingale spline
	 sequences on $I^d$ can easily be deduced from its
	 one-dimensional version as well. Indeed, for Banach spaces $X$ without
	 {\rm RNP} we
	 get, for any positive integer $k_1$, a non-convergent $X$-valued $k_1$-martingale
	 spline sequence
	 $(f_n^1)$ on $I$. Then, $f_n(x_1,\ldots,x_d) =
	 f_n^1(x_1)$ is a non-convergent $X$-valued $(k_1,\ldots,k_d)$-martingale spline
	 sequence on $I^d$ for
	 any choice of positive integers $k_2,\ldots,k_d$.

	 The main objective of this article is to prove the remaining assertions
	 \eqref{it:splines2}, \eqref{it:splines3} and the implication (a)$\implies$(b) in item \eqref{it:splines4}
	 for martingale spline sequences on $I^d$.
	The basic idea in the proof of \eqref{it:splines2} for $d=1$ (see
	\cite[Proposition~2.3]{PassenbrunnerShadrin2014}) 
	 is the pointwise bound
	\begin{equation}\label{eq:HLbound}
		\|P_n^k f(x)\| \leq C_k  \mathscr M_{\rm HL} f(x)
	\end{equation}
	of $P_n^k$ by the \emph{Hardy-Littlewood maximal function} 
	\begin{equation}\label{eq:HL}
		\mathscr M_{\rm HL}f(x) = \sup_{J\ni x} \frac{1}{|J|} \int_J
		\|f(y)\|\dif y,	
	\end{equation}
	where the supremum is taken over all intervals $J$ that contain the point $x$.
	This is enough to imply \eqref{it:splines2} for $d=1$ as it is a well
	known fact that $\mathscr M_{\rm HL}$ itself satisfies the weak type (1,1)
	bound
	\[
		| \{ \mathscr M_{\rm HL} f > t \}| \leq \frac{3}{t} \| f
		\|_{L^1_X},\qquad t>0.
	\]
	In dimensions $d>1$, by using this ad-hoc approach (see
	\cite[Proposition~3.3]{PassenbrunnerProchno2019}) one would need
	the \emph{strong maximal function} $\mathscr M_{\rm S} f(x)$ on the right hand side
	of \eqref{eq:HLbound}, where $\mathscr M_{\rm S}f(x)$ is defined by the same formula \eqref{eq:HL} as
	$\mathscr M_{\rm HL}f(x)$, but where
	the supremum is taken over all $d$-dimensional axis-parallel rectangles
	$J\subset I^d$ containing the point $x$.
	As a matter of fact, this is not enough to derive
	\eqref{it:splines2}, since the best possible weak type inequality for
	$\mathscr M_{\rm S}$ is true only in the Orlicz space $L(\log L)^{d-1}$ (see
	\cite{Guzman1973, Jessen1935, Saks1935}), which is a strict subset of $L^1$.

	Here we show how to employ the martingale spline structure,
	especially nestedness of atoms, to avoid the usage of the strong maximal
	function $\mathscr M_{\rm S}$ altogether 
	and replace it by an intrinsic maximal function that is (as we will
	show) of weak type $(1,1)$.
	This is crucial in the proof of the statements
	\eqref{it:splines2}, \eqref{it:splines3}, \eqref{it:splines4} 
	for any dimension $d$. Those statements are in full analogy to the martingale and
	one-dimensional spline results.
	The validity of \eqref{it:splines2} and \eqref{it:splines3} for martingale
	spline sequences on $I^d$ solves a problem stated in \cite{PassenbrunnerProchno2019}.
	
	The organization of this article is as follows.
	In Section~\ref{sec:prelims} we collect a few basic facts about vector
	measures needed in the sequel.
	In Section~\ref{sec:maximal}, we prove items \eqref{it:splines2} and
	\eqref{it:splines3} for martingale spline sequences on $I^d$
	(Proposition~\ref{prop:maximal} and
	Theorem~\ref{thm:aeconv}
	respectively).
	In Section~\ref{sec:convergence}, the implication
	(a)$\implies$(b) of item \eqref{it:splines4} is proved in this case
	(Theorem~\ref{thm:conv}) under the restriction that $\cup_n \mathscr
	F_n$ generates the Borel-$\sigma$-algebra on $I^d$.
	In Section~\ref{sec:convergence_general}, we show this assertion for
	general interval filtrations on $I^d$
	and give an explicit formula for the pointwise limit of martingale spline sequences.
\section{Preliminaries}
\label{sec:prelims}
We refer to the book \cite{DiestelUhl1977} by J. Diestel and J.J. Uhl
for basic facts on vector valued integration, martingales, vector measures and
the results that follow.

Let $\Omega$ be a set, $\mathscr A$ an algebra of subsets of $\Omega$ and
$(X,\|\cdot\|)$ a Banach space. A function $\nu:\mathscr A \to X$ is a \emph{(finitely
additive) vector measure} if, whenever $E_1,E_2\in\mathscr A$ are disjoint, we
have $\nu(E_1\cup E_2) = \nu(E_1) + \nu(E_2)$.
If, in addition, $\nu(\cup_{n=1}^\infty E_n) = \sum_{n=1}^\infty \nu(E_n)$ in the
norm topology of $X$ for all sequences $(E_n)$ of mutually disjoint members of
$\mathscr A$ such that $\cup_{n=1}^\infty E_n\in\mathscr A$, then $\nu$ is 
a \emph{countably additive vector measure}.
The \emph{variation} $|\nu|$ of a finitely additive vector measure $\nu$  is the set function
\begin{equation*}
	|\nu|(E) = \sup_\pi \sum_{A\in\pi} \|\nu(A)\|,
\end{equation*}
where the supremum is taken over all partitions $\pi$ of $E$ into a finite number
of mutually disjoint members of $ \mathscr A$.
If $\nu$ is a finitely additive vector measure, then the variation
$|\nu|$ is monotone and finitely additive.  
The measure $\nu$ is of \emph{bounded variation} if $|\nu|(\Omega) < \infty$.
If $\mu:\mathscr A \to [0,\infty)$ is a finitely additive measure and
	$\nu:\mathscr A\to X$ is
	a finitely additive vector measure, $\nu$ is \emph{$\mu$-continuous} if
	$\lim_{\mu(E)\to 0} \nu(E)=0$.
	If $\mu_1,\mu_2 : \mathscr A\to [0,\infty)$ are two finitely additive
		measures on $\mathscr A$, $\mu_1$ and $\mu_2$ are mutually
		\emph{singular} if for each $\varepsilon>0$ there exists a set
		$A\in\mathscr A$ so that
		\[
			\mu_1( A^c) + \mu_2(A) \leq \varepsilon.
		\]
		\begin{thm}[{Lebesgue decomposition of vector measures}]\label{thm:lebesgue}
Let $\mathscr A$ be an algebra of subsets of the set $\Omega$. Let $\nu :
\mathscr A\to X$ be a finitely additive vector measure of bounded variation. Let
$\mu :
\mathscr A\to [0,\infty)$ be a finitely additive measure.

	Then there exist unique finitely additive vector measures 
	of bounded variation $\nu_c,\nu_s$ so that
\begin{enumerate}
	\item $\nu = \nu_c + \nu_s$, $|\nu| = |\nu_c| + |\nu_s|$,
	\item $\nu_c$ is $\mu$-continuous,  
	\item $|\nu_s|$ and $\mu$ are mutually singular.
\end{enumerate}
\end{thm}

This theorem can be found in \cite[Theorem~9 on p.  31]{DiestelUhl1977}.
The following theorem is part of \cite[Theorem~2 on p. 27]{DiestelUhl1977} after
using \cite[Proposition 15 on p. 7]{DiestelUhl1977}.
\begin{thm}[Extension theorem]\label{thm:extension}
	Let $\mathscr A$ be an algebra of subsets of a set $\Omega$ and let
	$\mathscr F$ be the $\sigma$-algebra generated by $\mathscr A$.
	Let $\nu:\mathscr A\to X$ be a countably additive vector measure of
	bounded variation.

	Then, $\nu$ has a unique countably additive extension $\overline{\nu} :
	\mathscr F\to X$.
\end{thm}

\begin{defin}[{\cite[Definition 3, p. 61]{DiestelUhl1977}}]\label{def:rnp}
		A Banach space $X$ admits the \emph{Radon-Nikod\'{y}m property
	(RNP)} if
		for every measurable space $(\Omega,\mathscr F)$, for every
		positive, finite, countably additive measure $\mu$ on $(\Omega,\mathscr F)$ and for every
		$\mu$-continuous, countably additive vector measure $\nu$ of bounded variation,
		there exists a function $f\in L^1_X(\Omega,\mathscr F,\mu)$ such
		that
		\begin{equation*}
			\nu(A) = \int_A f\dif\mu,\qquad A\in\mathscr F.
		\end{equation*}
	\end{defin}

\section{Maximal functions of Tensor spline projectors}
\label{sec:maximal}
Let $d$ be a positive integer and let $(\mathscr F_n)= (\mathscr F_n^1\otimes
\cdots\otimes \mathscr F_n^d)$ be an interval filtration
on $I^d$ for some interval $I=(a,b]$ with $a<b$ and $a,b\in \mathbb R$.
Each $\sigma$-algebra $\mathscr F_n$ is then generated by a finite, mutually
disjoint family 
$\{ I_{n,i} : i\in \Lambda \}$, $\Lambda\subset \mathbb Z^d$, of 
$d$-dimensional rectangles of the form $I_{n,i} = \prod_{\ell=1}^d
(a_\ell,b_\ell]$ for some $a\leq a_\ell < b_\ell\leq b$. 
We write $\mathcal A(\mathscr F_n) = \{ I_{n,i} : i\in\Lambda\}$ to denote
this collection of atoms of the $\sigma$-algebra $\mathscr F_n$.
We assume that
$\Lambda$ is of the form $\Lambda^1\times \cdots\times \Lambda^d$ where for each
$\ell=1,\ldots,d$, $\Lambda^\ell$ is a finite set of consecutive integers and the
rectangles $I_{n,i}$ have the property that they are ordered in the same way as
$\mathbb R^d$, i.e., if $i,j\in\Lambda$ with $i_\ell < j_\ell$ then the
projection of $I_{n,i}$ onto the $\ell$th coordinate axis lies to the left of the
projection of $I_{n,j}$ onto the $\ell$th coordinate axis.
For $x\in I^d$, let $A_n(x)$ be the uniquely determined atom (rectangle)  $A\in\mathscr
F_n$ so that $x\in A$. For two atoms $A,B\in \mathscr F_n$, define
$d_n(A,B) := |i-j|_1$ if $A=I_{n,i}$ and $B=I_{n,j}$
and where $|w|_1 = \sum_{\ell=1}^d
|w_\ell|$ denotes the $\ell^1$ norm of the vector $w$.
If $U=\cup_\ell A_\ell$ and $V=\cup_\ell B_\ell$ are (finite) unions of atoms in
$\mathscr F_n$, we set $d_n(U,V) =
\min_{\ell,m} d_n(A_\ell,B_m)$.
Additionally, for a non-negative integer $s$, define
$A_{n,s}(x)$ to be the union of all atoms $A$ in $\mathscr F_n$ with
$d_n(A,A_n(x))\leq s$. Moreover, for a Borel set $B\subset I^d$, let $A_{n,s}(B)
= \cup_{x\in B} A_{n,s}(x)$.

For each $\ell=1,\ldots,d$, let $k_\ell$ be a positive integer.
Define the tensor product spline space of order $k= (k_1,\ldots,k_d)$ associated to $\mathscr
F_n$ as
\[
	S_n := S^{k_1}(\mathscr F_n^{1}) \otimes \cdots \otimes
	S^{k_d}(\mathscr F_n^{d}). 
\]
The space $S_n$ admits the tensor product B-spline basis
$(N_{n,i})_{i}$ defined by
\[
	N_{n,i} = N_{n,i_1}^1\otimes \cdots \otimes N_{n,i_d}^d,
\]
where $(N_{n,i_\ell}^\ell)_{i_\ell}$ denotes the B-spline basis of
$S^{k_\ell}(\mathscr F_n^{\ell})$ that forms a partition of unity.
The support $E_{n,i}=\supp N_{n,i}$  of
$N_{n,i}$ is composed of at most $k_1\cdots k_d$ neighbouring atoms of $\mathscr
F_n$. 
Consider the orthogonal projection operator $P_n= P_n^{k}$ onto $S_n$ with respect to
$d$-dimensional Lebesgue measure $|\cdot|=\lambda^d$.
Using the B-spline basis and its biorthogonal system $(N_{n,i}^*)$, 
the orthogonal projector $P_n$ is given by
\begin{equation}\label{eq:rep_Pn}
	P_n f =\sum_i \int_{I^d} fN_{n,i}\dif\lambda^d \cdot  N_{n,i}^{*},\qquad
	f\in L^1_X(I^d).
\end{equation}
The dual B-spline functions $N_{n,i}^*$ admit the following crucial geometric decay estimate 
\begin{equation}	
	\label{eq:mainestimate}
	| N_{n,i}^*(x) | \leq C \frac{ q^{d_n(E_{n,i}, A_n(x))} }{ |\conv
		(E_{n,i}\cup
	A_n(x))|}, \qquad x\in I^d,
\end{equation}
for some constants $C$ and $q\in [0,1)$ that depend only on  $k$, where $\conv(S)$
denotes the smallest, axis-parallel rectangle containing the set $S$.
This inequality was shown in
\cite[Theorem~1.2]{PassenbrunnerShadrin2014} for $d=1$ and if $d>1$,
\eqref{eq:mainestimate} is a consequence of the fact that $N_{n,i}^*$ is the
tensor product of one-dimensional dual B-spline functions.
Inserting this estimate in formula \eqref{eq:rep_Pn} for $P_n f$ and as $E_{n,i}$ 
consists of
at most $k_1\cdots k_d$ neighbouring atoms of $\mathscr F_n$, setting $C_k :=
C(k_1\cdots k_d) q^{-|k|_1}$, we get the
pointwise estimate
\begin{equation}
	\label{eq:estPn}
	\|P_n f(x)\| \leq C_k \sum_{A \in \mathcal A(\mathscr F_n)} b_n(q,
	\|f\|\dif\lambda^d, A, x),\qquad f\in L^1_X(I^d)
\end{equation}
introducing the expression
\begin{equation}\label{eq:defbn}
	b_n(q,\theta,A,x) = 
	\frac{q^{d_n(A,A_n(x))}}{|\conv(A \cup A_n(x))|}\theta(A),\qquad A\in\mathcal A(\mathscr F_n),  x\in I^d
\end{equation}
for a positive,  finitely additive measure $\theta$ on the algebra $\mathscr A =
\cup_m\mathscr F_m$.  
In view of inequality \eqref{eq:estPn}, it suffices to consider, instead of the maximal 
function of the projection operators $P_n$, 
the maximal functions given by
\begin{equation}\label{eq:max_fct}
	\mathscr M_{K} \theta(x) = \sup_{n\geq K} \sum_{A\in\mathcal A(\mathscr F_n)}
	b_n( q, \theta, A, x),
	\qquad x\in I^d
\end{equation}
for any positive integer $K$ and some fixed parameter $q\in[0,1)$.

If we abbreviate by $\mathscr Mf$ the maximal function $\mathscr M_1
(|f|\dif\lambda^d)$, we have the following weak type (1,1) result. 
\begin{prop}\label{prop:maximal}
	The maximal function $\mathscr M$ is of weak type
	(1,1), i.e. there exists a constant $C$ depending only on the dimension
	$d$ and on the parameter $q\in[0,1)$, so that we
	have the inequality
	\[
		| \{ \mathscr M f > t \}| \leq  \frac{C}{t}\|
			f\|_{L^1},\qquad t>0,\ f\in L^1(I^d).
	\]
\end{prop}
\begin{proof}
	Set $B=I^d$, $K=1$ and $\theta = |f|\dif\lambda^d$ in
	Theorem~\ref{thm:main_inequality} below and observe
	that the geometric series in equation \eqref{eq:main_inequality}
	converges.
\end{proof}

The following result about the maximal operators $\mathscr M_K$ is the focal
point in our investigations.
\begin{thm}\label{thm:main_inequality}
Let $(\mathscr F_n)$ be an interval filtration on $I^d$ and let $\theta$ be a
non-negative, finitely additive measure on the algebra $\mathscr A = \cup_n
\mathscr F_n$.

Then, for any Borel set $B\subset I^d$ and any positive integer $K$,
\begin{equation}\label{eq:main_inequality}
	| B \cap \{\mathscr M_{K}\theta > t\} |\leq \frac{C}{t} \cdot \sum_{s=0}^\infty
	q^{s/2} (s+1)^{d-1} \theta\big( A_{K,s}(B) \big),\qquad t>0
\end{equation}
for some constant $C$ depending only on $q$ and $d$.
\end{thm}
\begin{proof}
	Set $G_t = B \cap \{ \mathscr M_{K} \theta > t \}$ and let 
	$x\in G_t$. 
Then, there exists an
	index $n\geq K$ so that 
	\[
	\sum_{A\in\mathcal A(\mathscr F_n)} b_n(q,\theta,A,x) > t. \] 
	Letting $c=(2
	\sum_{\ell=0}^\infty \rho^\ell)^d = \big(2/(1-\rho)\big)^d<\infty$ with
	$\rho = q^{1/2}$, we obtain that there exists at
	least one atom
	$F$ of the $\sigma$-algebra $\mathscr F_{n}$ so that
	\begin{equation}\label{eq:lambda1}
			b_n(\rho,\theta, F, x)> t/c.
	\end{equation}
	Therefore, for $x\in G_t$, we choose $n_x<\infty$ to be the minimal index
	$n\geq K$ so that
	there exists an atom $F$ of $\mathscr F_{n}$ satisfying inequality
	\eqref{eq:lambda1}. We choose a particular 
	atom $F$ of $\mathscr F_{n_x}$
	with this property which will be denoted by $F_x$.
	The collection of atoms
		$\{ A_{n_x}(x) : x\in G_t \}$
	is nested and covers the set $G_t$. Thus, 
	it is possible to choose a countable subset $\Gamma\subset G_t$	
	such that the corresponding 
	collection 
		$\{ A_{n_x}(x) : x\in \Gamma\}$
		consists only of maximal and mutually disjoint atoms, 
		still covering the set $G_t$.
	Perform
	the following estimate using inequality \eqref{eq:lambda1}:
	\begin{equation}
		\label{eq:est1}
	\begin{aligned}
		|G_t| &\leq \sum_{x\in \Gamma} |A_{n_x}(x)| \leq
		\frac{c}{t} \cdot \sum_{x\in\Gamma} \rho^{d_{n_x}(F_x,
		A_{n_x}(x))} \theta(F_x) \\
		&= \frac{c}{t}\cdot  \sum_{s \geq 0}\rho^s \sum_{m\in\mathbb
		Z^d : |m|_1 = s}
		\Big( \sum_{x\in \Gamma_m} \theta(F_x) \Big),
	\end{aligned}	
	\end{equation}
	where for  $m\in\mathbb Z^d$, $\Gamma_m$ is the set of all
	$x\in \Gamma$ so that, if $A_{n_x}(x) = I_{n_x,i}$ and $F_x= I_{n_x,j}$
	for some $i,j\in\mathbb Z^d$, we have $i- j =m$.

	Next, we show that for each $m\in\mathbb Z^d$, the collection 
	$\{ F_x : x\in \Gamma_m\}$ consists of mutually disjoint sets.
	Assume the contrary, i.e. for some $m\in \mathbb Z^d$ there exist
two points $x,y\in \Gamma_m$ that are different from each other with $F_x\cap
	F_y\neq \emptyset$. 
	For definiteness, assume that $n_x\geq
	n_y$, and thus the nestedness of the
	$\sigma$-algebras $(\mathscr F_n)$ implies $F_x\subseteq F_y$.
	Assume that $i,i',j,j'\in\mathbb Z^d$ are such that
	\[
		I_{n_x,i} = A_{n_x}(x),\quad I_{n_y,i'} = A_{n_y}(y), \quad
		I_{n_x,j} = F_x,\quad I_{n_y,j'} = F_y.
	\]
	Since $x,y\in \Gamma_m$, we know that $i-j = m = i' - j'$.
	Therefore, since $\mathscr F_{n_x}$ is finer than $\mathscr F_{n_y}$ and
	by the inclusion $F_x \subseteq F_y$, we have
	\begin{equation}
		\label{eq:conv}
		\conv (F_x \cup A_{n_x}(x))\subseteq 
		\conv(F_y\cup
		A_{n_y}(x)) \subseteq 
		\conv (F_y \cup
		A_{n_y}(y)).
	\end{equation}
	Moreover, this and the definition of the distance $d_{n_y}$ implies
	\begin{equation}\label{eq:geometric}
		d_{n_y}(F_y, A_{n_y}(y)) \geq d_{n_y}(F_y,A_{n_y}(x)).
	\end{equation}
	Combining \eqref{eq:conv} and \eqref{eq:geometric} yields
	$b_{n_y}(\rho,\theta,
	F_y, x) \geq b_{n_y}(\rho,\theta,F_y,y)$; additionally, by definition of $n_y,
	F_y$  we have the inequality
	$b_{n_y}(\rho,\theta,F_y,y) > t/c$. Together, this implies
	\begin{equation*}
		b_{n_y}(\rho,\theta, F_y, x) > t/c.
		\end{equation*}
		As $n_x\geq K$ is the minimal index so that such an inequality
		at the point $x$ 
		is possible and $n_x\geq n_y$ we get that $n_x = n_y=:n$.
	Since $A_{n}(x) \cap A_{n}(y) =  \emptyset$ we know that in this case $i\neq i'$ and
	$x,y\in \Gamma_m$ implies $i-j =m = i'-j'$. Together, this
	yields $j\neq j'$ which means $F_x\cap F_y=\emptyset$, contradicting 
	the assumption $F_x\subseteq F_y$. Therefore, $F_x$ and $F_y$ are
	disjoint, concluding the proof of the fact that $\{ F_x : x\in
	\Gamma_m\}$ consists of
	mutually disjoint sets for each $m\in\mathbb Z^d$.

	If $(U_j)$ is a countable collection of disjoint members of $\mathscr A$
	and if $U\in \mathscr A$ with $\cup_{j=1}^\infty U_j
	\subset U$, then
		$\sum_{j=1}^\infty \theta(U_j) \leq \theta(U)$,
	since for finite sums this is clear by finite additivity and positivity
	of $\theta$ and
	the general case follows by passing to infinity.
	We apply this simple fact to the sum
	$\sum_{x\in\Gamma_m}\theta(F_x)$ with $U = A_{K,|m|_1}(B)$ to obtain
	from \eqref{eq:est1} 
	\begin{align*}
		|G_t| &\leq \frac{c}{t} \cdot \sum_{s=0}^{\infty} \rho^s 
		\theta\big(A_{K,s}(B) \big) \Big(\sum_{|m|_1 = s} 1\Big) \leq
		\frac{2^d c}{t}\sum_{s=0}^\infty \rho^s (s+1)^{d-1} \theta \big(
		A_{K,s}(B)\big),
	\end{align*}
	which is the conclusion of the theorem.
\end{proof}

Combining Proposition~\ref{prop:maximal} with the bound \eqref{eq:estPn} on the
operators $P_n$, we obtain that the
maximal function of the spline projectors $P_n$ also satisfies a weak type (1,1)
inequality
\begin{equation}\label{eq:weakPn}
	| \{ \sup_n \| P_n f \| > t\} | \leq \frac{C \|f\|_{L^1_X}
}{t}, \qquad t>0,\ f\in L^1_X(I^d),
\end{equation}
for some constant $C$ depending only on $k$. This proves Doob's inequality
\eqref{it:splines2} on page \pageref{it:splines2} for martingale spline
sequences on $I^d$. Indeed,
given a martingale spline sequence $(f_n)$ on $I^d$, 
apply \eqref{eq:weakPn} to the function $f=f_m$ for a fixed
positive integer $m$ and pass $m\to \infty$ to get
\eqref{it:splines2} for martingale spline sequences on $I^d$.

As a corollary, we have the following result about almost everywhere
convergence of $P_nf$ for $f\in L^1_X(I^d)$, proving \eqref{it:splines3} for
tensor spline projections.
\begin{thm}\label{thm:aeconv}
	Let $X$ be any Banach space and let $f\in L^1_X(I^d)$. Then, there
	exists $g\in L^1_X(I^d)$ such that
	\[
		P_n f \to g\qquad \text{$\lambda^d$-almost everywhere}.
	\]
\end{thm}
\begin{remark} (i) The proof of Theorem~\ref{thm:aeconv} follows along the same lines as the proof of
the one-dimensional case \cite[Theorem~3.2]{MuellerPassenbrunner2020} and uses standard arguments for
passing from a weak type maximal inequality of the form \eqref{eq:weakPn} to
almost everywhere convergence of $P_nf$ for $L^1$-functions $f$. For this argument,  a
dense subset of $L^1$ is needed, for which it is ``clear'' that pointwise
convergence takes place. In \cite[Lemma~3.1]{MuellerPassenbrunner2020}, for 
one-dimensional splines, this dense set is chosen to be the space of
continuous functions $C(\bar I)$ on the closure of the interval $I$.
For an arbitrary dimension $d$, we can use
 $C(\bar I)\otimes \cdots \otimes C(\bar I)$ as dense subset of $L^1$, for which it is
a consequence of the one-dimensional convergence result \cite[Lemma~3.1]{MuellerPassenbrunner2020}
and its tensor product structure that $P_n f$ converges pointwise for $f\in C(\bar
I)\otimes \cdots \otimes C(\bar I)$.

(ii) As in the one-dimensional case, the limit function $g$ in
Theorem~\ref{thm:aeconv} can be 
identified explicitly as the ($L^1_X$-extension of the) orthogonal projection of the function $f$ onto the
closure of $\cup_n S_n$, which, in the particular case that $\cup_n \mathscr
F_n$ generates the Borel-$\sigma$-algebra on $I^d$, coincides with the
function $f$.
\end{remark}

We also note another immediate corollary of Theorem~\ref{thm:main_inequality}
that will be used later.
\begin{cor}\label{cor:limsup}
Let $(\mathscr F_n)$ be an interval filtration on $I^d$ and let $\theta$ be a
non-negative, finitely additive measure on the algebra $\mathscr A = \cup_n
\mathscr F_n$.
Let $D\in \mathscr A$ be arbitrary and set 
\[
	L_t := \Big\{ x \in I^d : \limsup_n \sum_{A\in\mathcal A(\mathscr F_n)} b_n(q,\theta,A,x)
	> t\Big\}.
\]
Let $R$ be a non-negative integer.

If $B\subset D$ is a Borel set such that $A_{K, R}(B) \subset D$
for some $K$, we have
\begin{equation}\label{eq:localweaktype}
	| B \cap L_t  | \leq  \frac{C}{t} \Big( \theta(D) + \sum_{s>R}
	q^{s/2}(s+1)^{d-1}  \theta(I^d) \Big), \qquad t>0
\end{equation}
for some constant $C$ depending only on $d$ and $q$.
\end{cor}
\begin{proof}
	This just follows from Theorem~\ref{thm:main_inequality} by noting that
	$L_t \subset \{ \mathscr M_K \theta > t \}$ for any positive integer
	$K$.
\end{proof}

\begin{remark}
	Assume that in Corollary~\ref{cor:limsup}, the measure $\theta$ is a
	$\sigma$-additive Borel
	measure on $\bar I^d$ and replace the term $\theta(A)$ in the definition \eqref{eq:defbn} of
	$b_n$ by the term $\theta(\overline{A})$ with the closure $\overline{A}$
	of $A$ in $\bar I^d$. Then, the assertion of Corollary~\ref{cor:limsup}
	still holds if we replace $\theta(D)$ and $\theta(I^d)$ on the right
	hand side of \eqref{eq:localweaktype} by $\theta(\overline{D})$ and
	$\theta(\bar I^d)$ respectively. Indeed, the only modification in the
	proof of Theorem~\ref{thm:main_inequality} is that we have to replace
	$F_x$ by $\overline{F_x}$ in \eqref{eq:est1}, but this only gives an
	additional factor of $2^d$ on the right hand side of
	\eqref{eq:main_inequality} and
	\eqref{eq:localweaktype} since each point of $\bar I^d$ is contained in
	at most $2^d$ closures of disjoint rectangles.
\end{remark}

\section{The Convergence Theorem for dense filtrations}\label{sec:convergence}
In this section, we show the remaining implication (a)$\implies$(b) of item
\eqref{it:splines4} on page~\pageref{it:splines4} for martingale spline sequences on
$I^d$ in the case where $\mathscr A:= \cup_n\mathscr F_n$ generates the Borel-$\sigma$-algebra
on $I^d$. We restrict ourselves to this special setting in this section to present
the crucial arguments in a concise form. In order to lift the subsequent result
from this hypothesis, we use technical arguments in the spirit of  those in the
proof of the one-dimensional result \cite[Sections 4 and
6]{MuellerPassenbrunner2020}. This will be presented in detail in Section~\ref{sec:convergence_general}.

\begin{thm}\label{thm:conv}
Let $(\mathscr F_n)$ be an interval filtration on $I^d$ so  that $\mathscr A=\cup_n\mathscr F_n$
	generates the Borel-$\sigma$-algebra and let
	$X$ be a Banach space with RNP. 
	Let $(g_n)$ be an $X$-valued martingale spline sequence adapted to
	$(\mathscr F_n)$ 
	with $\sup_n \|g_n\|_{L^1_X} < \infty$.

	Then, there exists $g\in L^1_X(I^d)$ so that $g_n \to g $ almost
	everywhere with respect to Lebesgue measure $\lambda^d$.
\end{thm}
\begin{remark}
	As for martingales (see \cite{DiestelUhl1977}), the basic proof idea of this result is to
	define a vector measure $\nu$ based upon the martingale spline sequence
	$(g_n)$, whose absolutely continuous part with respect to Lebesgue
	measure $\lambda^d$ has a density $g\in L^1_X$ by the RNP of $X$, which
	is then the a.e. limit of $g_n$.
\end{remark}
\begin{proof}
	\textsc{Part I: } The limit operator $T$. \\
	For $f\in S_m$ and $n\geq m$,  since the operator $P_n$ is selfadjoint and using
the martingale spline property of the sequence $(g_n)$,
\begin{align*}
	\int_{I^d} g_n \cdot f \dif\lambda^d &= \int_{I^d} g_n\cdot P_m
	f\dif\lambda^d = \int_{I^d} P_m g_n \cdot f\dif
	\lambda^d = \int_{I^d} g_m
\cdot f\dif\lambda^d.
\end{align*}
This means in particular that for all $f\in \cup_m
S_m$, the limit of $\int_{I^d} g_n\cdot f\dif\lambda^d$
exists, so we define the linear operator 
\begin{equation*}
	T : \cup_m S_m \to X,\qquad  f\mapsto \lim_n
	\int_{I^d} g_n\cdot
	f \dif\lambda^d.
\end{equation*}
We can write $g_n$ in terms of this operator $T$. Indeed, by the martingale
spline property of $(g_n)$ again,
\begin{equation}\label{eq:gn}
\begin{aligned}
	g_n = P_n g_n &= \sum_{i} \int_{I^d} g_n N_{n,i}\dif\lambda^d\cdot
	N_{n,i}^*  \\
	&=\sum_{i} \lim_m \int_{I^d} g_m N_{n,i} \dif\lambda^d \cdot N_{n,i}^*
	=\sum_{i} (T N_{n,i}) N_{n,i}^{*}.
\end{aligned}
\end{equation}
By Alaoglu's theorem, we may choose a  subsequence $\ell_n$ such that the
bounded sequence of measures $\|g_{\ell_n}\|_X
\dif\lambda^d$ converges in the weak*-topology on the space of Radon measures on
the closure $\bar{I}^d$ of $I^d$
to some finite scalar measure $\mu$ on
$\bar{I}^d$, i.e. 
\[
	\lim_{n\to\infty} \int_{\bar{I}^d} f \|g_{\ell_n}\|  \dif\lambda^d =
	\int_{\bar{I}^d}
	f\dif\mu,\qquad f\in C(\bar{I}^d).
\]
For a fixed positive integer $m$, we then get
another subsequence of  $(\ell_n)$, again denoted by $(\ell_n)$, so that for each atom $A$ of $\mathscr F_m$,
the sequence $\|g_{\ell_n}\|\dif\lambda^d$ weak*-converges to some Radon measure
$\mu_A$ on the closure $\overline{A}$ of $A$ satisfying 
$\mu = \sum_{A \in\mathcal A(\mathscr F_m)} \mu_A$.
Each function $f\in S_m$ is continuous and a polynomial in the interior $A^\circ$ of each atom
$A\in\mathscr F_m$. Denote by $f_A$ the continuous function on the closure
$\overline{A}$ of $A$ that coincides with $f$ on $A^\circ$.
Then, for $\ell_n\geq m$ and $f\in S_m$
\begin{align*}
	\|Tf\| &= \Big\|\int_{I^d}f g_{\ell_n} \dif \lambda^d\Big\| \leq
	\int_{I^d} |f| \|
	g_{\ell_n}\| \dif \lambda^d\\
	&=\sum_{A\in\mathcal A(\mathscr F_m)} \int_{\overline{A}} |f_A|
	\|g_{\ell_n}\|\dif \lambda^d \rightarrow\sum_{A\mathcal A(\mathscr F_m)}
\int_{\bar{I}^d} |f_A| \dif\mu_{A} \\
&\leq \sum_{A\in\mathcal A(\mathscr F_m)} \int_{\bar{I}^d} \limsup_{s\to y}
	|f(s)| \dif\mu_A(y)
	= \int_{\bar{I}^d} \limsup_{s\to y} |f(s)| \dif\mu(y).
\end{align*}
For $f\in\cup_n S_n$ define
\begin{equation}\label{eq:strangenorm}
	\| f \| := \int_{\bar{I}^d} \limsup_{s\to y} |f(s)|\dif\mu(y),
\end{equation}
which is a seminorm on $\cup_n S_n$. As for $L^p$-spaces, we factor out the
functions $f\in\cup_n S_n$ with $\|f\|=0$ in order to get a norm. Then, 
denote by $W$  the completion of $\cup_n S_n$ in this norm 
and extend the operator $T$ to $W$ continuously.

\textsc{Part II: }Representing $T$ in terms of a vector measure $\nu$. \\
Let $Q = \prod_{\ell=1}^d (a_\ell, b_\ell]$ be an arbitrary atom of the $\sigma$-algebra
$\mathscr F_n$ for some positive integer $n$. 
Let $\ell\in \{1,\ldots,d\}$ be an arbitrary coordinate direction. If the order
of the polynomials $k_\ell$ in direction $\ell$ equals $1$ (piecewise constant
case), we set $f_m^\ell =
\charfun_{(a_\ell,b_\ell]}$ for $m\geq n$, which satisfies $f_m^\ell \in
S^{k_\ell}(\mathscr F_m^\ell)$.
If $k_\ell > 1$, we first choose an open interval $O$ and a closed interval $C$
(both in $I$) so that $C\subseteq (a_\ell, b_\ell] \subseteq O$ and $|O\setminus
C| \leq 1/m$. The sets $C$ and $O$ are chosen so that as many endpoints of $C$
and $O$ coincide with the corresponding endpoints of $(a_\ell,b_\ell]$ as
possible.
Then, let $f_m^\ell\in \cup_j S^{k_\ell}(\mathscr F_j^\ell)$ be a non-negative function that is bounded by
$1$ and satisfies
\[
	\supp f_{m}^\ell \subset O\qquad \text{and}\qquad f_m^\ell \equiv 1
	\text{ on } C.
\]
Such a function exists since $\mathscr A$ generates the Borel-$\sigma$-algebra
if one
additionally notices the facts that B-splines form a partition of unity and
have localized support.
If we define $f_m = f_m^1 \otimes \cdots \otimes f_m^d$,
the sequence $(f_m)$ is Cauchy
in $\cup_j S_j$ with respect to the norm in \eqref{eq:strangenorm}
and we let $I_Q$  be the limit in $W$ of the
sequence $(f_m)$ satisfying 
\[
	\|T I_Q\| = \lim_{m\to\infty} \|T f_m\| \leq
\mu(\overline{Q})
\]  (here, the closure of $Q$ is taken in $\bar I^d$).
This definition of $I_Q$ also has the property that if $Q$ is an atom in
$\mathscr F_n$ and $(Q_j)_{j=1}^\ell$ is a finite sequence of disjoint atoms
$Q_j$ in $\mathscr F_{n_j}$ with $n_j\geq n$ and  $Q=\cup_{j=1}^\ell Q_j$,
we have $I_{Q} = \sum_{j=1}^\ell I_{Q_j}$.
Therefore, if $\mathscr F_n\ni A = \cup_{j=1}^\ell Q_j$ for some disjoint atoms $(Q_j)_{j=1}^\ell$
in $\mathscr F_n$, it is well defined to set 
\[
	I_A = \sum_{j=1}^\ell I_{Q_j}\in W.
\]
Based upon that, we define the finitely additive vector measure $\nu$ on
$(I^d,\mathscr A)$
with values in $X$ by
\begin{equation}\label{eq:defnu}
	\nu(A) := T(I_{A}),\qquad A\in\mathscr A.
\end{equation}
This vector measure $\nu$ is of bounded variation, since if $\pi$ is a finite partition
of $I^d$ into sets of $\mathscr A$ and if $m<\infty$ is the minimal index
so that $A\in \mathscr F_m$ for all $A\in \pi$, we have
\[
	\sum_{A\in\pi} \|T(I_A)\| \leq \sum_{Q \in\mathcal A(\mathscr F_m)} 
	\|T(I_Q)\|\leq \sum_{Q \in\mathcal A(\mathscr F_m)}\mu(\overline{Q}) \leq
	2^d \mu(\bar{I}^d), 
\]
as each point in $\bar{I}^d$ is contained in at most $2^d$ closures of atoms of
$\mathscr F_m$. 

Observe that for all $f\in \cup_n S_n$, we have
\begin{equation}\label{eq:ext}
	\int_{I^d} f\dif\nu = T(f).
\end{equation}
Indeed, each $f\in\cup_n S_n$ can be approximated uniformly by linear combinations
of characteristic functions of atoms of the form $\chi_m :=\sum_{Q \in\mathcal A(\mathscr F_m)} \alpha_Q \charfun_Q$ as $m\to\infty$,
which then also has the property that 
$f_m:=\sum_{Q\in\mathcal A(\mathscr F_m)} \alpha_Q I_Q \to f$ in $W$ as
$m\to\infty$ and thus also  $Tf_m \to Tf$ in $X$ by the continuity of the
operator $T$. As, by definition \eqref{eq:defnu} of $\nu$, we have $ \int \chi_m
\dif\nu= Tf_m$, equation \eqref{eq:ext} follows by letting $m\to\infty$.

\textsc{Part III: }Conclusion. \\Continuing the calculation in equation
\eqref{eq:gn}, using the measure $\nu$ and \eqref{eq:ext},
\begin{equation}\label{eq:gn2}
	g_n = \sum_i \int_{I^d} N_{n,i}\dif\nu \cdot N_{n,i}^{*}.
\end{equation}
Apply Lebesgue's decomposition Theorem~\ref{thm:lebesgue} to the measure $\nu$
with respect to $\lambda^d$ to get two finitely
additive measures $\nu_c,\nu_s$ of bounded variation with
\begin{equation}\label{31-10-17-9}\nu = \nu_c + \nu_s, \end{equation}
where $\nu_c$ is $\lambda^d$-continuous and $|\nu_s|$ is singular to
$\lambda^d$. As $\lambda^d$ is countably additive, so is the
$\lambda^d$-continuous measure $\nu_c$ and by the extension theorem
(Theorem~\ref{thm:extension}) extends uniquely to  a countably additive vector
measure $\overline{\nu_c}$ on the Borel-$\sigma$-algebra on $I^d$, which, by the RNP of $X$ can be written as
$\dif\overline{\nu_c} = g\dif\lambda^d$ for some $g\in L^1_X$. Therefore, 
\[
	g_n = \sum_{i} \int_{I^d} N_{n,i} g\dif \lambda^d \cdot
	N_{n,i}^{*} + \sum_{i} \int_{I^d} N_{n,i}\dif\nu_s \cdot
	N_{n,i}^{*}.
\]
The first part on the right hand side of
this equation equals $P_n g$ for the $L^1_X$ function $g$ and this converges a.e. to $ g$
 by Theorem~\ref{thm:aeconv} and the remark following it.

The second part, denoted by $P_n\nu_s$, converges to $0$ almost everywhere,
which we will now show.
Let $t>0$ be  arbitrary  and define
\[ 
	G_{t} := 
\{ y\in I^d : \limsup_n \|P_n\nu_s(y)\| >t\}.
\]
Then, let $\varepsilon>0$ be arbitrary and choose
$D\in \mathscr A$ 
with the property
\[
	\lambda^d(D^c)+|\nu_s|(D)\leq \varepsilon,
\] 
 which is possible since 
$|\nu_s|$ is singular to $\lambda^d$. 
By \eqref{eq:estPn}, replacing $\|f\|\dif \lambda^d$ with $|\nu_s|$,
\begin{align*}
	\|P_n\nu_s(y)\| &\leq C_k \sum_{A \in \mathcal A(\mathscr F_n)} b_n(q,
	|\nu_s|, A, y)
\end{align*}
for some constants $C_k$ and $0<q<1$ depending only on $k$ with $b_n$ as in
\eqref{eq:defbn}. 
Therefore, $G_t \subset L_{t/C_k}$ with 
\[
	L_u = \Big\{ y\in I^d : \limsup_n \sum_{A \in\mathcal A(\mathscr F_n)} b_n(q,
|\nu_s|, A, y) > u \Big\}.
\]
We apply Lemma~\ref{lem:limsup} below (with $Y=I^d$) to the measure $\theta=|\nu_s|$ on
$\mathscr A$ and the set $D$ to
get, for any $u>0$, the estimate 
\[ 
	|L_u| = |D^c\cap L_u| + |D\cap L_u|  \leq \varepsilon + C\varepsilon/u.
	\]
Since this is true for any $\varepsilon>0$, we obtain $|L_u| = 0$ for any
$u>0$. Thus,
\[
	|\{ y\in I^d: \limsup_n \|P_n\nu_s(y)\| >0\}| =\Big|\bigcup_{r=1}^\infty
	G_{1/r}\Big| \leq \Big|\bigcup_{r=1}^\infty L_{1/(C_k r)} \Big|=
	\lim_{r\to\infty} |L_{1/(C_k r)}| = 0,
\]
which completes the proof of the theorem. 
\end{proof}

%

\section{The convergence theorem for arbitrary filtrations}
\label{sec:convergence_general}
	Now we discuss the necessary modifications in the proof of Theorem~\ref{thm:conv}
	when the interval filtration $(\mathscr F_n)$ on $I^d$ is allowed to be
	arbitrary. Assume for some Banach space $X$ with RNP, $(g_n)$ is an $X$-valued martingale spline
	sequence adapted to $(\mathscr F_n)$ with $\sup_n \|g_n\|_{L^1_X} <
	\infty$.

	Part I of the proof of Theorem~\ref{thm:conv} does not use the density
	of the filtration $(\mathscr F_n)$ in $I^d$, which means that
	we get an operator $T:\cup_n S_n \to X$ and a finite measure $\mu$ on
	$\bar I^d$ satisfying
	\begin{equation}
		\label{eq:T_mu}
		\| Tf \| \leq \int_{\bar I^d} \limsup_{s\to y}
		|f(s)|\dif\mu(y),\qquad f\in \cup_n S_n.
	\end{equation}
	The operator $T$ is then extended continuously to the completion $W$ of $\cup_n S_n$
	w.r.t. the norm on the right hand side of \eqref{eq:T_mu}.
	With the aid of this operator, the martingale spline sequence $(g_n)$ is
	written as
	\begin{equation}\label{eq:rep_gn_T}
		g_n = \sum_i (TN_{n,i}) N_{n,i}^*.
	\end{equation}

	We  distinguish the analysis of the convergence of
	$g_n(y)$ depending on in which coordinate direction the filtration $(\mathscr F_n)$ 
	is dense at the point $y$.
	To this end, for $\ell=1,\ldots,d$, we define $\Delta_n^\ell\subset \bar I$ to be the set of
	all endpoints of atoms in the $\sigma$-algebra $\mathscr F_n^\ell$.
	Next, let $U^\ell$ be the complement (in $\bar I$) of the set of all accumulation points
	of $\cup_n \Delta_n^\ell$. Note that $U^\ell$ is open (in $\bar I$), thus it can be
	written as a countable union of disjoint open intervals $(U^\ell_j)_j$. 
	Let 
	\begin{align*}
		B_j^\ell = \{ a\in \partial U_j^\ell : \text{ there is no
			sequence of points in $U_j^\ell \cap (\cup_n
		\Delta_n^\ell)$ } 
		\text{that converges to $a$} \}
	\end{align*}
	and define $V_j^\ell := U_j^\ell \cup B_j^\ell$ and $V^\ell := \cup_j
	V_j^\ell$.

\begin{lem}\label{lem:limsup}
Let $(\mathscr F_n)$ be an interval filtration on $I^d$ and let $\theta$ be a non-negative, 
finitely additive and finite
	measure on $\mathscr A$.
	For $\varepsilon>0$, let $D\in\mathscr A$ 
	 with $\theta(D)\leq\varepsilon$
	and 
	\[
		L_t := \Big\{ x \in I^d : \limsup_n \sum_{A\in\mathcal A(\mathscr F_n)} b_n(q,\theta,A,x)
	> t\Big\}.
	\]

	Then, there exists a finite constant $C$, depending only on $q$ and on
	$d$
	 so that 
	\[
		|D\cap L_t \cap Y|\leq \frac{C  \varepsilon}{t},\qquad t>0,
	\]
	with  $Y = (V^1)^c \times \cdots \times(V^d)^c$.
\end{lem}
\begin{proof}
	We shrink the set $D$
	properly to then apply Corollary~\ref{cor:limsup}. This is done as
	follows. Since
	$D\in\mathscr A$, we can write it as
$D = \cup_{j=1}^L Q_j$ 
for disjoint atoms $(Q_j)$ of some $\sigma$-algebra $\mathscr F_n$.
For each $j$, we have $Q_j = Q_j^1 \times \cdots\times Q_j^d$ for some intervals $Q_j^\ell$,
$\ell = 1,\ldots, d$.
Assume without restriction that for all $\ell\in\{1,\ldots,d\}$, the interior of
the interval
$Q_j^\ell$ contains at least two points from
$(V^\ell)^c$, since otherwise we would have $|Q_j \cap L_t \cap Y|\leq |Q_j\cap
Y| =0$.
Fix $\ell \in \{1,\ldots,d\}$, set $\eta =\varepsilon/(tL |I|^{d-1} d)$ and define the
interval $J^\ell \subset Q_j^\ell$ such that 
\begin{enumerate}
	\item $Q_j^\ell\setminus J^\ell$ has two connected components and in
		each one 
		there exists a point of $(V^\ell)^c$
		that has positive distance
		to $J^\ell$ and to the boundary of $Q_j^\ell$,
	\item $ | (Q_j^\ell \setminus J^\ell) \cap  (V^\ell)^c| \leq \eta $.
\end{enumerate}
This is possible since $Q_j^\ell\cap (V^\ell)^c$ contains at least two points.
Then, set $Q_j' = J^1\times \cdots \times J^d$ and $B = \cup_{j=1}^L Q_j'$ and
we get, by the choice of $\eta$,
\begin{equation}\label{eq:D_and_B_notdense}
	|D\cap L_t\cap Y| \leq | (D\setminus B)\cap Y| + |B\cap L_t| \leq
	\varepsilon/t + |B\cap L_t|.
\end{equation}
	Choose the positive integer $R$ sufficiently large so that 
	\[
		\sum_{s>R} (s+1)^{d-1} q^{s/2} \theta(I^d) \leq \varepsilon.
	\]
	Then, there exists an integer $K$ so that
	$A_{K,R}(B)\subset D$, which is true 
	by construction of $B$.
	Apply now Corollary~\ref{cor:limsup} to
	get $|B\cap L_t| \leq C\varepsilon/t$, which together with
	\eqref{eq:D_and_B_notdense} implies the assertion of the lemma.
\end{proof}
\begin{remark}
	Assume that in Lemma~\ref{lem:limsup}, the measure $\theta$ is a
	$\sigma$-additive Borel
	measure on $\bar I^d$ and replace the term $\theta(A)$ in the definition \eqref{eq:defbn} of
	$b_n$ by the term $\theta(\overline{A})$ with the closure $\overline{A}$
	of $A$ in $\bar I^d$. Then, the assertion of Lemma~\ref{lem:limsup}
	still holds with an additional factor of $2^d$ on the constant $C$,
	since the same is true for Corollary~\ref{cor:limsup}.
\end{remark}

	For a point $y = (y^1,\ldots,y^d)\in I^d$, each
	coordinate $y^\ell$ is either contained in some set $V^\ell_{j_\ell}$ or
	in $(V^\ell)^c$. After rearranging the coordinates, we assume that 
	$y \in F$,
where $F = F^1\times \cdots \times F^d$ with 
$F^\ell = V_{j_\ell}^\ell$ if $\ell \leq s$ and $F^\ell = (V^\ell)^c$ if $\ell
> s$ for some $s\in \{0,\ldots,d\}$. 
We want to split $g_n(y) = \sum_i (T N_{n,i}) N_{n,i}^*(y)$ into the parts where $T$ acts on functions
	restricted to the set $F_\delta$ for $\delta\in \{ 0, 1 \}^d$
	with $F_\delta = E^1 \times \cdots \times E^d$ where $E^\ell = F^\ell$
	if $\delta_\ell=0$ and $E^\ell = (F^\ell)^c$ if $\delta_\ell=1$.
	In order to construct elements in $W$ that correspond to the functions
	$N_{n,i}\charfun_{F_\delta}$, we need the following lemma.

	\begin{lem}\label{lem:restriction}
		For any $\ell\in \{1,\ldots,d\}$, let $f\in S^{k_\ell}(\mathscr
		F_n^\ell)$ for some $n$. For any interval $V_{j}^\ell$, there
		exists a sequence $(h_m)$ of functions $h_m\in
		S^{k_\ell}(\mathscr F_m^\ell)$, open intervals $O_m$ and closed 
		intervals $C_m$ (both in $\bar I$) satisfying
		\begin{enumerate}
			\item $O_m \to V_j^\ell$ as $m\to\infty$,
			\item $\supp h_m\subset O_m$,
			\item $h_m \equiv f$ on $C_m\cap I$,
			\item The closure of $O_m\setminus C_m$ converges to the
				empty set as $m\to \infty$.
		\end{enumerate}
	\end{lem}
	\begin{proof}
		Without loss of generality, assume that $f= N_{n,i}^\ell$ for
		some integer $i$. For $m\geq n$, we can write
		\[
			N_{n,i}^\ell = \sum_{r} \lambda_{m,r} N_{m,r}^\ell,
		\]
		where the absolute value of each coefficient
		$\lambda_{m,r}$ is $\leq 1$.  Set
		\[
			h_m = \sum_{r\in \Lambda_m} \lambda_{m,r} N_{m,r}^\ell,
		\]
		where the set $\Lambda_m$ is defined to contain precisely those indices
		$r$ so that the support of $N_{m,r}^\ell$ intersects $V_{j}^\ell$
		but the (Euclidean) distance between the support of $N_{m,r}^\ell$ and 
		$\partial U_{j}^\ell \setminus B_{j}^\ell$ is positive.
		The function $h_m$ is then contained in $S^{k_\ell}(\mathscr F_m^\ell)$
		and satisfies $|h_m|\leq 1$. With this setting, the support of $h_m$
		is contained in $O_m$ for some open interval $O_m$ and $h_m \equiv
		N_{n,i}^\ell$ on some closed interval $C_m\subset O_m$.
		Since the endpoints of $V_j^\ell$ are accumulation points
		of $\cup_n \Delta_n^\ell$ or endpoints of $I$, the intervals $O_m$ and $C_m$ can be chosen to satisfy
		items (1) and (4).
	\end{proof}

	Let now $(h_{j,m}^\ell)_m$ be the sequence of functions from
	Lemma~\ref{lem:restriction} corresponding to a
	function $f^\ell\in S^{k_\ell}(\mathscr F_{n_\ell}^\ell)$ for some
	positive integer $n_\ell$ and the set $V_j^\ell$.
	\begin{enumerate}
		\item If $E^\ell = V_{j_\ell}^\ell$, set $h_m^\ell =
			h_{j_\ell,m}^\ell$.
		\item If $E^\ell = (V_{j_\ell}^{\ell})^c$, set $h_m^\ell =
			f^\ell - h_{j_\ell,m}^\ell$.
		\item If $E^\ell = V^\ell$, set $h_m^\ell = \sum_{j=1}^m
			h_{j,K_m}^\ell$.
		\item If $E^\ell = (V^\ell)^c$, set $h_m^\ell = f^\ell -  \sum_{j=1}^m
			h_{j,K_m}^\ell$.
	\end{enumerate}
	Then, define $h_m = h_m^1\otimes \cdots \otimes h_m^d$.
 Since $\cup_{j\geq m} \overline{V_j^\ell}$ tends to the empty set
	as $m\to \infty$ for each $\ell$, and due to the properties guaranteed by
	Lemma~\ref{lem:restriction} of the functions $(h_m^\ell)$, if $K_m$ is
	chosen sufficiently large, $h_m\in S_{K_m}$ is Cauchy
	in the Banach space $W$ and its limit will be denoted by $(f^1
	I_{E^1}) \otimes \cdots \otimes (f^d I_{E^d})$.
	If $f^\ell = N_{n,i_\ell}^{\ell}$ is some B-spline function for all
	$\ell$ and some positive integer $n$, we will also write $N_{n,i}
	I_{F_\delta}$ for this limit in $W$, which (by \eqref{eq:T_mu}) satisfies
	\begin{equation}\label{eq:est_NI}
		\| T(N_{n,i} I_{F_\delta}) \| = \lim_m \| Th_m\| \leq \liminf_m
		\mu(\overline{\supp h_m}) \leq \mu\big( F_\delta \cap \overline{\supp
			N_{n,i}}\big).
	\end{equation}

	This construction allows us to decompose the martingale spline sequence $g_n$ into

	\[
		g_n =\sum_{\delta\in \{0,1\}^d} g_{n,\delta}, \qquad \text{with
		}
		g_{n,\delta} = \sum_i T(N_{n,i} I_{F_\delta})
		N_{n,i}^* \text{ for } \delta\in \{0,1\}^d.
	\]
	We treat the sequence $(g_{n,\delta})_n$ for each fixed $\delta\in \{0,1\}^d$
	separately.

	\medskip
	\textsc{Case 1: }We begin by considering the case where one of the first $s$ coordinates
	of $\delta$ equals one.  Without restriction assume that the first
	 coordinate of $\delta$ equals one. 
	 Write 
	 $N_{n,i}^* = N_{n,i_1}^{1*} \otimes N_{n,i_2}^{>1*}$, with
	 $i=(i_1,i_2)$ for an integer $i_1$ and a $(d-1)$-tuple of integers
	 $i_2$, thus $g_{n,\delta}$ can be written as 
	 \begin{equation}\label{eq:splitting_delta_prime_not_zero}
		 g_{n,\delta}(y_1,y_2) = \sum_{i_2} \Big(\sum_{i_1} T(
		 N_{n,i}I_{F_\delta})
		 N_{n,i_1}^{1*}(y_1) \Big)
		N_{n,i_2}^{>1*}(y_2), \qquad (y_1,y_2) \in F.
	\end{equation}
	Fix $y_1 \in U_{j_1}^1$ and  $t>0$.
	Let $\varepsilon>0$ and denote by $A_n^1(y_1)$ the atom in $\mathscr F_n^1$ that contains the
	point $y_1$. Then, $\beta:= \inf_n | A_n^1(y_1) | >0$ since $U_{j_1}^1$
	does not contain accumulation points of $\cup_n \Delta_n^1$.
	 Choose an open interval $O \supseteq V_{j_1}^1$ so that
		$\mu\big( (O\setminus V_{j_1}^1) \times \bar{I}^{d-1} \big) 
		\leq \varepsilon \mu(\bar I^d)$.
	 Then, choose $M$ sufficiently large so that for all $n\geq M$, we have 
		 $q^{d_n (A_n^1(y_1), B_n)} \leq \varepsilon$
	 for all atoms $B_n$ in $\mathscr F_n^1$ with $B_n\cap O^c\neq \emptyset$.
	 This is possible since the endpoints of $V_{j_1}^1$ are accumulation
	 points of $\cup_n \Delta_n^1$.
	 Split the sum over $i_1$ in
	 \eqref{eq:splitting_delta_prime_not_zero} into indices $i_1$ so that
	 $\overline{\supp N_{n,i_1}^1}\subseteq O$ and its complement and use the
	 geometric decay estimate \eqref{eq:mainestimate} for the dual B-splines
	 $N_{n,i_1}^{1*}$ and $N_{n,i_2}^{>1*}$ and estimate
	 \eqref{eq:est_NI}.
	 With the measures 
	 \[
		 \theta_1(A) = \frac{1}{\beta} \mu\big( (O\setminus V_{j_1}^1) \times A\big), 
		 \qquad \theta_2(A) = \frac{\varepsilon}{\beta} \mu\big( \bar I
		\times A\big)
	 \]
	satisfying $\max \{ \theta_1(\bar I^{d-1}), \theta_2(\bar I^{d-1}) \}
	 \leq \varepsilon \mu(\bar I^d)/\beta$
	 and the notation $\mathscr F_n^{>1} = \mathscr F_n^2\otimes \cdots \otimes
	 \mathscr F_n^d$, we then
	 obtain for 
	 $n\geq M$ 
	 \[
		 \| g_{n,\delta}(y_1,y_2) \| \leq C \sum_{\text{$A\in\mathcal A( \mathscr
			 F_n^{>1})$}} \big(b_n(q,\theta_1, A, y_2) +
		 b_n(q,\theta_2,A,y_2)\big),
	 \]
	 where the expressions $b_n(q,\theta,A,y_2)$ are defined as in
	 \eqref{eq:defbn}, but with $\theta(A)$ replaced by $\theta(\overline A)$.
	Here and in the following, the letter $C$ denotes a
	constant that depends only on $k,d,q$ and that may change from line to line.
	 Then, applying Corollary~\ref{cor:limsup} (using also the remark
	 succeeding it) in dimension $d-1$ with
	 $B=D=\bar I^{d-1}$, we estimate
	 \begin{align*}
		| \{ y_2 : \limsup_n \| g_{n,\delta} (y_1,y_2) \| >t \}|
		&\leq C \frac{ \varepsilon \mu(\bar I^d)}{t \beta}.
	\end{align*}
	We have this inequality for any $\varepsilon>0$, which implies 
	$| \{ y_2 : \limsup_n \| g_{n,\delta} (y_1,y_2) \| >t \}| =0$. 
	As this is true for any $t>0$ and any $y_1\in U_{j_1}^1$, we
	get $g_{n,\delta} \to 0$ almost everywhere on $F$.
	
	\medskip
	\textsc{Case 2:} Next, consider the case where $\delta\neq 0$ but the first $s$
	coordinates of $\delta$ equal $0$.
	  Write $N_{n,i}^* = N_{n,i_1}^{\leq s*} \otimes
	 N_{n,i_2}^{>s*}$ where
	 $i=(i_1,i_2)$ for an $s$-tuple of integers $i_1$ and a $(d-s)$-tuple of integers
	 $i_2$, thus, $g_{n,\delta}$ can be written as 
	 \begin{align*}
		 g_{n,\delta} (y_1,y_2)  &= \sum_{i_2} \Big( \sum_{i_1}
		 T( N_{n,i} I_{F_\delta})  N_{n,i_1}^{\leq s*}(y_1)\Big)
		N_{n,i_2}^{>s*}(y_2), \qquad (y_1,y_2)\in F.
	\end{align*}
	Denote by $A_m^{\leq s}(y_1)$ 
	the atom $A$ in $\mathscr F_m^1 \otimes \cdots \otimes \mathscr F_m^s$
	with $y_1\in A$.
	If we fix $y_1 \in U_{j_1}^1\times \cdots \times
	 U_{j_s}^s$, we know that 
	$\beta := \inf_m |A_m^{\leq s}(y_1)|>0$.
	Next, define $Y= F^{s+1} \times \cdots \times F^d$ and $Z = E^{s+1}
	\times \cdots \times E^d$. Moreover, let $\mathscr F_m^{>s}= \mathscr
	F_m^{s+1} \otimes \cdots \otimes \mathscr F_m^d$ and define the measure 
	$\theta(A) = \mu(\bar{I}^s \times (A\cap Z))$.
	 Observe that
	$\theta(Y)=0$, since $E^\ell \cap F^\ell = \emptyset$ for some $\ell >s$
	by the form of $\delta$.
	Using estimate \eqref{eq:mainestimate} for the dual B-spline functions
	and
	estimate \eqref{eq:est_NI} bounding the operator $T$ in terms of $\mu$, 
	\begin{align*}
		\| g_{n,\delta}(y_1,y_2) \| &\leq \frac{C}{\beta} \sum_{\text{$A\in\mathcal A(\mathscr F_n^{>s})$}} b_n(q,\theta,A,y_2)
	\end{align*}
	 where the expression $b_n(q,\theta,A,y_2)$ is defined as in
	 \eqref{eq:defbn}, but with $\theta(A)$ replaced with
	 $\theta(\overline{A})$. 
	Approximate $Y$ by a sequence of sets $Y_m \in \mathscr F_m^{>s}$ with
	$Y_m \to Y$. Then, for each $\varepsilon>0$, there exists a positive
	integer $m(\varepsilon)$ with $| Y \setminus Y_{m(\varepsilon)} | \leq
	\varepsilon$ and $\theta(Y_{m(\varepsilon)}) \leq \varepsilon$.
	For $t>0$, apply Lemma~\ref{lem:limsup} (and the remark
	succeeding it) in
	dimension $d-s$ with
	$D=Y_{m(\varepsilon)}$ to deduce
	\begin{equation}\label{eq:Lt1}
		|L_t \cap Y| \leq | Y_{m(\varepsilon)}\cap L_t \cap Y| + |
		Y\setminus Y_{m(\varepsilon)}| \leq \frac{C\varepsilon}{t} +
		\varepsilon
	\end{equation}
	with
	\[
		L_t =  \Big\{ y_2 \in I^{d-s} : \limsup_n \sum_{\text{ $A\in\mathcal A(\mathscr F_n^{>s})$}} b_n(q,\theta,A,y_2) > t\Big\}.
	\]
	Since \eqref{eq:Lt1} holds for any $\varepsilon>0$, we obtain $|L_t\cap
	Y| = 0$ for any $t>0$, which gives that for any fixed $y_1$,
	$g_{n,\delta}(y_1,y_2)$ converges to $0$ a.e. in $y_2\in Y$.
	Summarizing and combining this with Case~1 for $\delta$,
	we have $g_{n,\delta} \to 0$ a.e. on $F$ as $n\to\infty$ if
	one of the coordinates of $\delta$ equals $1$.

	\medskip
	\textsc{Case 3: }It remains to consider the case $g_{n,0}$, i.e. the choice
	$\delta=0$. \\
	For each $\ell \leq s$, the B-splines $(N_{n,r}^\ell)_r$ whose supports
	intersect $V_{j_\ell}^\ell$  can be indexed in such a way that for each
	fixed $r$, the function $N_{n,r}^\ell \charfun_{V_{j_\ell}^\ell}$
	converges uniformly to a function $\bar N_r^\ell$ as $n\to\infty$ (cf. \cite[Section
	4]{MuellerPassenbrunner2020}). This is the case since the interior of
	$V_{j_\ell}^\ell$ does not contain any accumulation points of $\cup_n
	\Delta_n^\ell$. Depending on whether the endpoints of $V_{j_\ell}^\ell$
	can be approximated from inside of $V_{j_\ell}^\ell$ by points in $\cup_n \Delta_n^\ell$,
	there are different possibilities for the index set $\Lambda^\ell$ of the
	functions $(\bar N_r^\ell)_{r\in\Lambda^\ell}$. It can either be finite,
	infinite on one side or bi-infinite.

	We have the following biorthogonal functions to $(\bar
	N_r^\ell)_r$ that admit the same geometric decay estimate
	\eqref{eq:mainestimate} than the dual
	B-spline functions $N_{n,r}^{\ell*}$. This result is similar to
	\cite[Lemma~4.2]{MuellerPassenbrunner2020}.

		\begin{lem}\label{lem:limit_dual_bspline}
		Let $\ell\in\{1,\ldots,d\}$. For each $r\in \Lambda^\ell$,
		the sequence $N_{n,r}^{\ell*}$
		converges uniformly on each atom of $\mathscr A^\ell=\cup_n \mathscr F_n^\ell$
		contained in $V_{j_\ell}^\ell$ to some
		function $\bar N_{r}^{\ell*}$ satisfying the estimate
		\begin{equation}\label{eq:dualbarestimate}
					| \bar N_r^{\ell *} (y)| \leq
					C \frac{q^{d(A(y),E_r )}} {| \conv
					(A(y)\cup E_r) |},\qquad y\in
					U_{j_\ell}^\ell,
				\end{equation}
				denoting by $A(y)$ the atom of 
				$\mathscr A^\ell$ that contains the point $y$,
				by $E_r$ the support of $\bar N_r^\ell$
				and by $d(A(y),E_r)$ the number of atoms in
				$\mathscr A^\ell$ between $A(y)$ and
				$E_r$.
	\end{lem}
	\begin{proof}
		Fix the index $r\in\Lambda^\ell$, the point $y\in U_{j_\ell}^\ell$ and
		$\varepsilon>0$.
		Since
		$r\in\Lambda^\ell$ is fixed, the support $E_{n,r}$ of
		$N_{n,r}^\ell$ intersects $U_{j_\ell}^\ell$ for some
		index $n$ and we know that $\beta = \inf_m |E_{m,r}| > 0$.
		Additionally, set $\gamma = |A(y)|$. Without restriction, we
		assume that $\beta,\gamma\leq 1$.
		Next, we choose $L$ sufficiently large so that $L q^{L} \leq \varepsilon
		\beta\gamma$ and, for any positive integer $n$, $d_n(A_n(y),
		E_{n,r})\leq L$. Moreover, choose an open interval $O\supseteq
		V_{j_\ell}^\ell$ satisfying $|O\setminus U_{j_\ell}^\ell| \leq \varepsilon
		\beta\gamma/L$.
		Based on that, choose
		$M$ sufficiently large so that each of the intervals $(\inf O, y)$ and
		$(y,\sup O)$ contains at least $L$ points of $\Delta_M^\ell$ and
		so that, for indices $\nu$ with $|\nu-r|\leq 2L$ we have
		\begin{equation}\label{eq:infty}
			\| N_{n,\nu}^\ell - N_{m,\nu}^\ell
			\|_{L^\infty(U_{j_\ell}^\ell)} \leq \varepsilon
			\beta\gamma/L,\qquad m,n\geq M.
		\end{equation}
		For $n\geq m \geq M$, expand the
		function $N_{m,r}^{\ell*}$ in the basis $(N_{n,\nu}^{\ell*})_\nu$ as
		\begin{equation}
			N_{m,r}^{\ell*} = \sum_\nu \alpha_{r\nu} N_{n,\nu}^{\ell*}.
		\end{equation}
		The coefficients $\alpha_{r\nu}$ are bounded by a constant
		independently of $r,\nu$ and $m,n$ as we will now see. To this end,
		we use the geometric decay inequality \eqref{eq:mainestimate}
		for the dual B-spline functions $N_{m,r}^{\ell*}$ to obtain 
		\begin{align*}
			|\alpha_{r\nu}| &= \Big| \int_I N_{m,r}^{\ell*}
			N_{n,\nu}^{\ell}\dif\lambda\Big|  
			\leq C\sum_{ \text{ $A\in\mathcal A(\mathscr F_m^\ell)$} }
			\frac{q^{d_m(A,E_{m,r})}}{ | \conv
				(A \cup E_{m,r})|} \int_A 
			N_{n,\nu}^\ell\dif\lambda  \\
			&\leq C\sum_{ \text{ $A\in\mathcal A(\mathscr F_m^\ell)$}}
			\frac{q^{d_m(A,E_{m,r})}}{ | \conv
				(A\cup E_{m,r})|} |A| 
			\leq C\sum_{ \text{ $A\in\mathcal A(\mathscr F_m^\ell$)}}
			q^{d_m(A,E_{m,r})}
			\leq C.
		\end{align*}
		Denoting $f_\nu = N_{m,\nu}^\ell - N_{n,\nu}^\ell$, whose
		absolute value is bounded by $1$,
		\begin{align*}
			\delta_{r\nu} &= \int_I N_{m,r}^{\ell*}
			N_{m,\nu}^{\ell}\dif\lambda = 
			\int_I
			N_{m,r}^{\ell*} N_{n,\nu}^\ell \dif\lambda + \int_I
			 N_{m,r}^{\ell*} f_\nu\dif\lambda 
			 = \alpha_{r\nu} + \int_{I}
			 N_{n,r}^{\ell*} f_\nu\dif\lambda.
		\end{align*}
		For indices $\nu$ with $|\nu - r|\leq 2L$, we now estimate this
		last integral, by decomposing it into the integrals
		$I_1,I_2,I_3$ over
		$U_{j_\ell}^\ell$, $O\setminus U_{j_\ell}^\ell$ and $O^c$, respectively.
		By estimate \eqref{eq:infty} and the fact that
		the integral of $N_{n,r}^{\ell*}$ is smaller than a constant $C$
		by \eqref{eq:mainestimate}, the integral $|I_1|$ can be bounded by
		$C\varepsilon\beta\gamma/L$. For the second integral, we use
		the fact that the integrand is bounded by $C /\beta$ and the
		measure estimate for $O\setminus U_{j_\ell}^\ell$ to deduce
		$|I_2|\leq C \varepsilon\gamma / L$.
		For the remaining integral $I_3$, we note that on $O^c$,
		the function $N_{n,r}^{\ell*}$ is bounded by $C q^L/\beta$,
		which, together with estimate
		\eqref{eq:mainestimate} and the choice of $L$ gives $|I_3|\leq
		C\varepsilon\gamma/L$.

		Summarizing, 
		\[
			|\alpha_{r\nu} - \delta_{r\nu}|\leq C\varepsilon\gamma/L,\qquad
			|\nu-r|\leq 2L.
		\]
		This can be used to estimate the difference between
		$N_{m,r}^{\ell*}$ and $N_{n,r}^{\ell*}$ for $n\geq m\geq M$ pointwise as follows
		\begin{align*}
			| N_{m,r}^{\ell*}(y) - N_{n,r}^{\ell*}(y) | &=
			\Big|\sum_{\nu} (\alpha_{r\nu} - \delta_{r\nu}) N_{n,\nu}^{\ell*}(y)\Big|
			\leq C \varepsilon + \sum_{\nu : |\nu-r| >2L}
			|\alpha_{r\nu}
			N_{n,\nu}^{\ell*}(y) |,
		\end{align*}
		by using the bound $|N_{n,\nu}^{\ell*}(y)| \leq C/\gamma$.
		By the choice of $L$, the inequality $|\nu - r| > 2L$ yields 
		$d_n(A_n(y), E_{n,\nu})>L$ and thus,
		the geometric decay estimate for $N_{n,\nu}^{\ell*}$, the
		boundedness of $\alpha_{r\nu}$ and the choice of $L$ implies
		that the latter sum is bounded by $C\varepsilon$. This, in turn, leads to the estimate
		$| N_{m,r}^{\ell*}(y) - N_{n,r}^{\ell*}(y)| \leq C\varepsilon$
		and thus the convergence of $N_{n,r}^{\ell*}(y)$,
		which is uniform in $A(y)$ since all the estimates
		above only depend on $A(y)$ and not on the particular point $y$.
		Now, estimate \eqref{eq:dualbarestimate} follows from the corresponding
		estimate of $N_{n,r}^{\ell*}$ by letting $n\to\infty$.
	\end{proof}
	
	Write $F = Z\times Y$ with $Z = F^1 \times \cdots \times F^s = V_{j_1}^1
	\times \cdots \times V_{j_s}^s$ and
	$Y= F^{s+1}\times \cdots \times F^d$. 
	For an $s$-tuple of integers $i_1 = (r_1,\ldots,r_s)$ and a
	$(d-s)$-tuple of integers $i_2 = (r_{s+1},\ldots,r_d)$, set
	$N_{m,i_1}^{\leq s} I_Z = N_{m,r_1}^1 I_{F^1} \otimes \cdots \otimes
	N_{m,r_s}^s I_{F^s}$ and $N_{n,i_2}^{>s}I_Y = N_{n,r_{s+1}}^{s+1}
	I_{F^{s+1}} \otimes \cdots \otimes N_{n,r_d}^d I_{F^d}$.
	The uniform convergence of $N_{m,r_\ell}^{\ell}\charfun_{V_{j_\ell}^\ell}$ to
	$\bar N_{r_\ell}^\ell$ for $\ell\leq s$ as $m\to \infty$ implies that for fixed $n$ and $i_1$, the sequence
		$(N_{m,i_1}^{\leq s} I_Z \otimes N_{n,i_2}^{>s} I_Y)$
	converges in $W$ to some element as $m\to \infty$, which we denote by
	$\bar N_{i_1}^{\leq s} \otimes N_{n,i_2}^{>s} I_Y$.
	By the continuity of $T$, we also have $T( N_{m,i_1}^{\leq s}I_Z \otimes N_{n,i_2}^{>s} I_Y) \to T(\bar
		N_{i_1}^{\leq s} \otimes N_{n,i_2}^{>s} I_Y)$
	in $X$ as $m\to \infty$.
	Using the expressions $T(\bar
		N_{i_1}^{\leq s} \otimes N_{n,i_2}^{>s} I_Y)$ and the dual
		functions $\bar N_{i_1}^{\leq s*} = \bar N_{r_1}^{1*} \otimes
		\cdots\otimes \bar N_{r_s}^{s*}$ to $\bar
		N_{i_1}^{\leq s}$ given by
		Lemma~\ref{lem:limit_dual_bspline},
	define
	\begin{equation}\label{eq:un}
		u_{n} = \sum_{i_1,i_2} T (
		\bar N_{i_1}^{\leq s} \otimes N_{n,i_2}^{>s} I_{Y} ) (\bar
		N_{i_1}^{\leq s*} \otimes
		N_{n,i_2}^{>s*}).
	\end{equation}
	Next, we show that the sequence $(g_{n,0})$ and the sequence $(u_n)$ have
	the same a.e. limit on $F$. Indeed, for fixed $y_1 \in U_{j_1}^1\times \cdots
	\times U_{j_s}^s$, the difference of those two functions has
	the form
	\begin{align*}
		g_{n,0}(y_1,\cdot) - u_n(y_1,\cdot) &= \sum_{i_2} N_{n,i_2}^{>s*}
		\Big[ \sum_{i_1} T \big( ( N_{n,i_1}^{\leq s} I_Z - \bar
			N_{i_1}^{\leq s} ) \otimes N_{n,i_2}^{>s} I_Y\big)
		N_{n,i_1}^{\leq s*}(y_1) \\
		& \qquad +	\sum_{i_1} T(\bar N_{i_1}^{\leq s}
	\otimes N_{n,i_2}^{>s}I_Y) \big(N_{n,i_1}^{\leq s*}(y_1) - \bar
N_{i_1}^{\leq s} (y_1) \big) \Big].
	\end{align*}
	Denote $\mathscr F_n^{>s} = \mathscr F_n^{s+1}\otimes \cdots \otimes
	\mathscr F_n^d$.
	Using \eqref{eq:mainestimate} for $N_{n,i_2}^{>s*}$ and $N_{n,i_1}^{\leq
	s*}$,
	Lemma~\ref{lem:limit_dual_bspline}, 
	the uniform boundedness and the localized support of $\bar
	N_{i_1}^{\leq s}$, and the bound \eqref{eq:est_NI} of the operator $T$ in terms of the
	measure $\mu$, we obtain for all $\varepsilon>0$ an 
	index $M$ so that for $n\geq M$
	\begin{equation}\label{eq:diff}
		\| g_{n,0}(y_1,y_2) - u_n(y_1,y_2) \| \leq
		\sum_{\text{$A\in\mathcal A(\mathscr F_n^{>s})$}} b_n(q,\theta,A,y_2),  
	\end{equation}
	where $\theta$ is the measure given by $\theta(A) =
	\varepsilon\mu\big(\bar I^s \times
	(A \cap Y)\big)$ and the expression $b_n(q,\theta,A,y_2)$ is defined as
	in \eqref{eq:defbn}, but with $\theta(A)$ replaced with
	$\theta(\overline{A})$.
	By Corollary~\ref{cor:limsup} (and the remark succeeding it) with
	$B=D=\bar I^{d-s}$ we obtain
	$|L_t| \leq C\theta(\bar I^{d-s})/t \leq C\varepsilon \mu(\bar I^d)/t$
	with
	\[
		L_t = \{ y_2 \in I^{d-s} : \limsup_n \sum_{\text{$A\in\mathcal A(\mathscr F_n^{>s})$}} b_n(q,\theta,A,y_2)>t \}.
	\]
	This implies, using also \eqref{eq:diff}, 
	\[
		| \{ y_2\in I^{d-s} : \limsup_n \| g_{n,0}(y_1,y_2) -
	u_n(y_1,y_2) \| > t\} | = 0
	\]
	for any $t>0$, i.e., $g_{n,0}$ and $u_n$ have the same a.e. limit on $F$.

	Therefore, in order to identify the a.e. limit of $(g_{n,0})$ on $F$ (which, by
	Cases $1$ and $2$, is also the a.e. limit of $(g_n)$), we
	identify the a.e. limit of $(u_n)$.
	Similar to Part II in the proof of Theorem~\ref{thm:conv}, we want to
	construct, for each $i_1$, a vector measure $\nu_{i_1}$ on $\mathscr
	A^{>s} = \cup_n \mathscr F_n^{>s}$
	based on the expressions  $T(\bar N_{i_1}^{\leq s} \otimes
	N_{n,i_2}^{>s} I_Y)$.
	The aim is to have, for each B-spline function $N_{n,i_2}^{>s}$, the
	representation
	\begin{equation}\label{eq:measure_T}
		\int N_{n,i_2}^{>s} \dif\nu_{i_1} = T(\bar N_{i_1}^{\leq s}
		\otimes N_{n,i_2}^{>s} I_Y).
	\end{equation}
To this end, let $Q = \prod_{\ell=s+1}^d (a_\ell,
b_\ell]$ be an atom of the $\sigma$-algebra
	$\mathscr F_n^{>s}$ for some positive integer $n$.
For the definition of the measure $\nu_{i_1}(Q)$, we approximate the characteristic
function $\charfun_{Q}$ of $Q$ by spline functions $f_m^{s+1} \otimes \cdots \otimes f_m^d$
contained in $\cup_j
\big(S^{k_{s+1}}(\mathscr F_j^{s+1}) \otimes \cdots \otimes S^{k_d}(\mathscr
F_j^d)\big)$, which will be done as follows.
	Let $\ell\in \{s+1,\ldots,d\}$. If the order
	of the polynomials $k_\ell$ in direction $\ell$ equals $1$ (piecewise
	constant case), we set $f_m^\ell =
	\charfun_{(a_\ell,b_\ell]}$ for $m\geq n$, which satisfies $f_m^\ell \in
	S^{k_\ell}(\mathscr F_m^\ell)$.
	If $k_\ell > 1$, we apply the following construction of the
	approximation $f_m^\ell$ of the characteristic function of the interval
$(a_\ell,b_\ell]$. 

If $a_\ell$ is contained in the countable set $\cup_j \partial U_j^\ell$ and if
$a_\ell$ is not an endpoint of $I$, we choose $c \in (a_\ell, a_\ell+1/m)$ that is
not contained in $\cup_j \partial U_j^\ell$. Otherwise, set $c=a_\ell$.
Similarly, if $b_\ell$ is contained
in $\cup_j \partial U_j^\ell$ and if $b_\ell$ is not an endpoint of $I$, we
choose $d\in (b_\ell,b_\ell+1/m)$ that is not contained in $\cup_j \partial
U_j^\ell$. Otherwise, set $d=b_\ell$.
	Put 
	\[
		J(x) = \begin{cases}
			V_j^\ell, &\text{if }x\in U_j^\ell, \\
			\emptyset, &\text{otherwise,}
		\end{cases}
	\]
	and define the interval
	\[
	J = \Big((c,d] \setminus J(c) \Big) \cup J(d),
	\]
which has the property that $J\cap (V^\ell)^c = (c,d] \cap
(V^\ell)^c$.
We then choose a closed interval $C$ and an open interval $O$ (both in $I$) with
$C\subseteq
J\subseteq O$ and the property $|O\setminus C| \leq 1/m$.
The sets $C$ and $O$ are chosen so that as many endpoints of $C$
and $O$ coincide with the corresponding endpoints of $(c,d]$ as
possible.
Then, let $f_m^\ell\in \cup_j S^{k_\ell}(\mathscr F_j^\ell)$ be a non-negative function that is bounded by
	$1$ and satisfies
	\[
		\supp f_{m}^\ell \subseteq O \qquad \text{and}\qquad f_m^\ell
		\equiv 1 \text{ on } C.
	\]
	This is possible since if $c$ or $d$ are endpoints of $J$,
	they are contained in $\big(\cup_j \overline{U_j^\ell})^c$ and thus 
	can be approximated from both sides with grid points $\cup_j
	\Delta_j^\ell$. Otherwise, the endpoints of $J$ are also endpoints of
	some set $V_j^\ell$, which are accumulation points of
	$\cup_j\Delta_j^\ell$ as well.

Then, define $f_m = f_m^{s+1} \otimes \cdots \otimes f_m^d$ which gives, for each
index $i_1$,
a Cauchy sequence $\bar N_{i_1}^{\leq s} \otimes f_m I_Y$ in $W$. Its limit
will be written as $\bar N_{i_1}^{\leq s}\otimes  (I_Q\cdot I_Y)$.
Then, continuing in a similar fashion as in Part~II of the proof of
Theorem~\ref{thm:conv}, we  make sense of the expression $T\big(\bar
N_{i_1}^{\leq s} \otimes (I_A \cdot I_Y)\big)$ for any $A\in\mathscr A^{>s}$
and define the measure $\nu_{i_1}(A) = T\big(\bar
N_{i_1}^{\leq s} \otimes (I_A \cdot I_Y)\big)$ for $A\in \mathscr A^{> s}$ whose
total variation satisfies
$|\nu_{i_1}| (I^{d-s})  \leq 2^{d-s} \mu(\overline{\supp \bar
			N_{i_1}^{\leq s}}
		\times Y)$.
Additionally, for any B-spline function $N_{n,i_2}^{>s}$, we have equation
\eqref{eq:measure_T}.
	Now, as in Part III of the proof of Theorem~\ref{thm:conv}, denoting by
	$w_{i_1}$ the Radon-Nikod\'{y}m density of the absolutely continuous
	part of $\nu_{i_1}$ with respect to Lebesgue measure $\lambda^{d-s}$, 
	\[
		u_n(y_1,y_2) = \sum_{i_1}\bar N_{i_1}^{\leq s*}(y_1) (P_n^{>s}
		\nu_{i_1})(y_2)\to g(y_1,y_2) := \sum_{i_1}\bar N_{i_1}^{\leq
		s*}(y_1) w_{i_1}(y_2)
	 \]
	 as $n\to\infty$ for almost every $(y_1,y_2)\in F$. Using the estimate
	 from Lemma~\ref{lem:limit_dual_bspline} for $\bar N_{i_1}^{\leq s*}$
	 and the above estimate for the total variation of the measures
	 $\nu_{i_1}$, we obtain that $\| g \|_{L^1_X(F)} \leq C\cdot \mu(F)$.

	Thus, we have proven the following theorem:
\begin{thm}\label{thm:conv_general}
Let $(\mathscr F_n)$ be an interval filtration on $I^d$ 
	and let
	$X$ be a Banach space with RNP. 
	Let $(g_n)$ be an $X$-valued martingale spline sequence adapted to
	$(\mathscr F_n)$ 
	with $\sup_n \|g_n\|_{L^1_X} < \infty$.

	Then, there exists $g\in L^1_X(I^d)$ so that $g_n \to g $ almost
	everywhere with respect to Lebesgue measure $\lambda^d$.
\end{thm}
\begin{remark}
	Employing the notation developed in this section, we emphasize that the
	pointwise limit $g$ 
	has the explicit representation
\[
g(y_1,y_2) := \sum_{i_1}\bar N_{i_1}^{\leq
		s*}(y_1) w_{i_1}(y_2),\qquad (y_1,y_2)\in F,
\]
where $\bar N_{i_1}^{\leq s*}$ are the functions given by
Lemma~\ref{lem:limit_dual_bspline} corresponding to $F^1\times \cdots \times F^s
= V_{j_1}^1 \times \cdots \times V_{j_s}^s$ and the function $w_{i_1}$ is the
Radon-Nikod\'{y}m density of the absolutely continuous part (w.r.t Lebesgue
measure $\lambda^{d-s}$) of the measure
$A\mapsto T\big(\bar N_{i_1}^{\leq s} \otimes (I_A \cdot I_Y)\big)$
with $Y = (V^{s+1})^c \times \cdots \times (V^d)^c$.
\end{remark}
\subsection*{Acknowledgements } 
The author is  supported by the Austrian Science Fund FWF, project P32342. 

\bibliographystyle{plain}
\bibliography{convergence}

\end{document}